\documentclass[11pt]{amsart}
\usepackage[mathscr]{eucal}

\usepackage{amsmath,amssymb,amsfonts,amsthm, color}

\newtheorem{thm}{\textbf{Theorem}}[section]
\newtheorem{lem}{\textbf{Lemma}}[section]

\newcommand{\Z}{\mathbb Z}
\newcommand{\R}{\mathbb R}

\newcommand{\be}{\begin{equation}}
\newcommand{\ee}{\end{equation}}
\newcommand{\und}{\;\mbox{ and }\;}
\newcommand{\nn}{\nonumber}
\newcommand{\ber}{\begin{eqnarray}}
\newcommand{\eer}{\end{eqnarray}}

\newcommand{\diam}{\mathsf{diam}\,}
\newcommand{\first}{\mathsf{first}}
\newcommand{\last}{\mathsf{last}}

\newcommand\floor[1]{{\lfloor{#1}\rfloor}}
\newcommand\ceil[1]{{\lceil{#1}\rceil}}

\begin{document}
\title{On three sets with nondecreasing diameter}
\author{Daniel Bernstein, David J. Grynkiewicz and Carl R. Yerger}

\begin{abstract}Let $[a,b]$ denote the integers between $a$ and
$b$ inclusive and, for a finite subset $X\subseteq \Z$,  let ${\diam}(X)=\max(X)-\min(X)$. We write $X<_p\,Y$ provided  $\max(X)<\min(Y)$. For a positive
integer $m$, let $f(m,m,m;2)$ be the least integer $N$ such that any $2$-coloring $\Delta: [1, N]\rightarrow \{0,1\}$ has
three monochromatic $m$-sets $B_1,\,B_2,\,B_3\subseteq [1,N]$ (not necessarily of the same color) with
$B_1<_p\, B_2 <_p\, B_3$ and  ${\diam}(B_1)\leq {\diam}(B_2)\leq
{\diam}(B_3)$. Improving upon upper and lower bounds of Bialostocki, Erd\H os and Lefmann, we show that
$f(m,m,m;2)=8m-5+\lfloor\frac{2m-2}{3}\rfloor+\delta$ for $m\geq 2$, where $\delta=1$ if $m\in \{2,5\}$ and $\delta=0$ otherwise.\end{abstract}

\maketitle

\section{Introduction}
For $a,\,b\in \R$,
 we let $[a,b]$ denote the set of integers between $a$ and $b$ inclusive.
For finite subsets $X,\,Y\subseteq \Z$, the \emph{diameter}
of $X$, denoted by ${\diam}(X)$, is defined as $\max(X)-\min(X)$.
Moreover, we say that $X<_p\,Y$ if and only if $\max(X)<\min(Y)$, meaning all the elements of $X$ come before any element from $Y$. For
positive integers $t,\,m_1,\,m_2,\,\ldots m_t,\,r\in\Z^+$, let
$f(m_1,m_2,\ldots, m_t;r)$ be the least integer $N$ such that, for
every $r$-coloring $\Delta: [1, N]\rightarrow [0,r-1]$ of the
integers $[1,N]$, there exist $t$ subsets $B_1,\,B_2,\,\ldots,
B_t\subseteq [1,N]$ with
\begin{itemize}
\item[(a)] each $B_i$ monochromatic, i.e., $|\Delta(B_i)|=1$ for $i=1,\ldots,t$,
\item[(b)] $|B_i|=m_i$ for
$i=1,2,\ldots,t$
\item[(c)] $B_1<_p\, B_2 <_p\, \ldots <_p\, B_t$, and
\item[(d)]${\diam}(B_1)\leq {\diam}(B_2)\leq\ldots\leq {\diam}(B_t)$.
\end{itemize}
A collection of
monochromatic sets $B_i$ that satisfy (b), (c) and (d) is called a
\emph{solution} to $p(m_1,m_2,\ldots, m_t;r)$.

The function $f(m_1,m_2,\ldots, m_t;r)$, and the related function
$f^*(m_1,m_2,\ldots, m_t;r)$ defined as $f(m_1,m_2,\ldots, m_t;r)$
but requiring the inequalities in (d) to be strict, have been
studied by previous authors. Bialostocki, Erd\H{o}s and Lefmann
first introduced $f(m,m,\ldots, m;r)$ in \cite{bial}, where they
determined that $f(m,m;2)=5m-3$, that $f(m,m;3)=9m-7$ and that
\be\label{oldbounds}8m-4\leq f(m,m,m;2)\leq 10m-6,\ee as well as
giving  asymptotic bounds for $t=2$. The problem was motivated in
part by zero-sum generalizations in the sense of the
Erd\H{o}s-Ginzburg-Ziv Theorem \cite{egz} \cite[Theorem
10.1]{davidbook} (see \cite{baginski} \cite{bial}
\cite{davidfourcolorzs} for a short discussion of zero-sum
generalizations, including definitions). Subsequently, Bollob\'{a}s,
Erd\H{o}s, and Jin obtained improved results for $m=2$, showing that
$4r-\log_2 r+1\leq f^*(2,2;r)\leq 4r+1$ and $f^*(2,2;2^k)=4\cdot
2^k+1$, as well as giving improved asymptotic bounds for $t$ and $r$
when $m=2$. The value of $f(m,m;4)$ was determined to be $12m-9$ in
\cite{davidfourcolor},
%
the off diagonal cases (when not all $m_i=m$) are introduced in  \cite{carl}, and
%
%
%
other related Ramsey-type problems can also be found in \cite{alon}
\cite{brown} \cite{davidandy} \cite{arobertson} \cite{arobertson2}
\cite{schultz}.

The goal of this paper is to improve the estimates from \eqref{oldbounds} to the first exact value for more than two sets. Indeed, we will show that both the upper and lower bounds of Bialostocki, Erd\H{o}s and Lefmann can be improved, resulting in the value $$f(m,m,m;2)= 8m-5+\lfloor\frac{2m-2}{3}\rfloor+\delta\quad\mbox{
for $m\geq 2$},$$ where $\delta=1$ if $m\in \{2,5\}$ and $\delta=0$ otherwise.

\section{Determination of $f(m,m,m;2)$}

Let $\Delta:X\rightarrow C$ be a coloring of a finite set $X$ by a
set of colors $C$. Let $c\in C$ and $Y\subseteq X$. Let
$x_1<x_2<\ldots<x_n$ be the integers colored by $c$ in $Y$. Then,
for integers $i$ and $j$ such that $1\leq i\leq j\leq n$, we use
the notation $\first_i^j(c,Y)$ to denote
$\{x_i,\,x_{i+1},\,\ldots,\,x_j\}$, which is the set consisting of the $i$-th through $j$-th smallest elements of $Y$ colored by $c$.  Likewise,
$\first_i(c,Y)=x_i$ is the $i$-th smallest element colored by $c$ in $Y$,  and $\first(c, Y)=\first_1(c,Y)=\min (\Delta^{-1}(c)\cap Y)$ is the first element colored by $c$ in $Y$. Similarly,
we define
$\last_i^j(c,Y)=\{x_{n-i+1},\,x_{n-i},\,\ldots,\,x_{n-j+1}\}$ to be the set consisting of the $i$-th through $j$-th largest elements of $Y$ colored by $c$,
$\last_i(c,Y)=x_{n-i+1}$ to be the $i$-th  largest element of $Y$ colored by $c$, and $\last(c,Y)=\last_1(c,Y)=\max(\Delta^{-1}(c)\cap Y)$ to be the last element of $Y$ colored by $c$. For the
sake of simplicity, a coloring $\Delta:[1,\,N]\rightarrow C$ will
be denoted by the string
$\Delta(1)\Delta(2)\Delta(3)\cdots\Delta(N)$, and $x^i$ will be
used to denote the string $xx\ldots x$ of length $i$. Hence
$\Delta :[1,\,6]\rightarrow \{0,1\}$, where $\Delta ([1,\,2])=\{0\}$,
$\Delta(3)=1$, and $\Delta ([4,6])=\{0\}$, may be represented by the
string $\Delta [1,6]=0^210^3$.

The following technical lemma will help us control the possible
$2$-colorings of $[1,3m-2]$.

\begin{lem}\label{lem2.1}
Let $m\geq 2$, let $\Delta: [1,3m-2]\rightarrow\{0,1\}$ be a $2$-coloring and let $B_1\subseteq [1,3m-2]$ be a monochromatic $m$-subset with ${\diam}(B_1)\geq
2m-2$ satisfying the following additional extremal constraints: \begin{itemize}
\item[(a)] $\max B_1:=3m-2-\beta$ is minimal, where $\beta\in [0,m-1]$;
\item[(b)] $\diam(B_1):=2m-2+\alpha$ is minimal subject to (a) holding,  where $\alpha\in [0,m-1-\beta]$. \end{itemize}
Suppose $B_1$ exists and $\Delta(B_1)=\{1\}$. Then, letting $R=[1,3m-2-\beta]$, one of the following holds.

\begin{itemize}
\item[(i)]
\begin{enumerate}
\item[$\bullet$] $\beta\leq m-2$ and  $|\Delta^{-1}(0)\cap R|\geq m$
\item[$\bullet$] $\Delta R=1^{m-1-\beta-\nu}H_00H_11^{1+\nu}$,
where $\mu,\,\nu\geq 0$ are integers
\item[$\bullet$] $\beta=m-1-\alpha$ or $\nu=\mu=0$
\item[$\bullet$] $H_1$ is a string of
length $m-2-\beta+\mu$
with exactly $\mu$ $1$'s and exactly $m-2-\beta$ $0$'s
\item[$\bullet$] $H_0$ is a string of length $m-1+\beta-\mu$ containing exactly $m-1-\alpha-\mu$ $1$'s and exactly $\alpha+\beta$ $0$'s
\end{enumerate}

\medskip

\item[(ii)]
\begin{enumerate}
\item[$\bullet$] either $\beta<m-1-\alpha$ or $\beta=m-1$
\item[$\bullet$]
$\Delta R=0^{m-\alpha-\beta-1}1H_21^{m-\beta}$
\item[$\bullet$]  $H_2$ is a string of length $m-2+\beta+\alpha$
\item[$\bullet$] if $\beta\leq m-2$, then $|\Delta^{-1}(0)\cap R|\geq m$
\item[$\bullet$] if $\alpha>0$, then $\beta\geq 1$ and $H_2$ contains exactly $\beta-1$ $1$'s
\end{enumerate}

\medskip

\item[(iii)]
\begin{enumerate}
\item[$\bullet$] $\beta\geq \alpha$
\item[$\bullet$] $|\Delta^{-1}(0)\cap R|<m$
\item[$\bullet$]
$\first_m(1,R)\leq 3m-3-\beta-\alpha$
\item[$\bullet$] $|\Delta^{-1}(0)\cap [\first_1(1,R),\first_m(1,R)]|\leq \beta$
\end{enumerate}
\end{itemize}
\end{lem}

\begin{proof}
 Note (a) and (b) imply \be\label{minB1}\min B_1=\max B_1-\diam B_1=3m-2-\beta-(2m-2+\alpha)=m-\alpha-\beta.\ee
Let $$\eta=3m-3-\beta-\last_2(1,R)\geq 0$$ be
the number of integers colored by $0$ between $\last_1(1,R)$ and $\last_2(1,R)$. Let $$\nu=3m-3-\beta-\last(0,R)\geq 0$$
be the number of integers strictly between $3m-2-\beta$ and
$\last(0,R)$ that are colored by $1$.
We continue with three claims.

\subsection*{Claim A} If $|\Delta^{-1}(1)\cap R|>m$, then $\last_2(1,R)\leq 3m-3-\beta-\alpha$ and $\eta\geq \alpha$.

If we have $\last_2(1,R)\geq 2m-2+\first(1,R)$,
then
$B'_1=\first_1^{m-1}(1,R)\cup \{\last_2(1,R)\}$ will be a monochromatic $m$-subset, in view of the hypothesis $|\Delta^{-1}(1)\cap R|>m$, with $\max B'_1<\max B_1$ and $\diam B'_1=\last_2(1,R)-\first(1,R)\geq 2m-2$, contradicting the maximality condition (a) for $B_1$. Therefore we may instead assume $\last_2(1,R)\leq 2m-3+\first(1,R)\leq 3m-3-\beta-\alpha,$ where the final inequality follows from $\first(1,R)\leq \min B_1$ and \eqref{minB1}, and now $\eta\geq \alpha$ follows from the definition of $\eta$, completing the claim.

\subsection*{Claim B} If $|\Delta^{-1}(0)\cap R|<m$, then either $\beta=m-1$ and (ii) holds or else (iii) holds.

If $\beta=m-1$, then $\alpha=0$ in view of $\alpha\in [0,m-1-\beta]$. Moreover, $\min B_1=1$ by \eqref{minB1}, and now (ii) is easily seen to hold. Therefore we may assume $\beta\leq m-2$, in which case \be\label{enough1s} |\Delta^{-1}(1)\cap R|=3m-2-\beta-|\Delta^{-1}(0)\cap R|\geq 2m-|\Delta^{-1}(0)\cap R|\geq m+1,\ee where we have utilized the claim hypothesis for the final inequality. Thus Claim A implies \be\label{wolf1}\last_2(1,R)\leq 3m-3-\beta-\alpha\quad\und\quad \eta\geq \alpha.\ee Moreover, since $\first_m(1,R)\leq \last_2(1,R)$ by \eqref{enough1s}, we see that \eqref{wolf1} also implies $$\first_m(1,R)\leq 3m-3-\beta-\alpha.$$
By the hypothesis of the claim, we have \be\label{waterspout}|\Delta^{-1}(1)\cap [m-\alpha-\beta,3m-2-\beta]|\geq 2m-1+\alpha-|\Delta^{-1}(0)\cap R|\geq m+\alpha.\ee
We must have \be\label{morew} \Delta([1,m-1-\beta-\eta])\subseteq \{0\},\ee for if $\first(1,R)\leq m-1-\beta-\eta$, then $B'_1=\first_1^{m-1}(1,R)\cup \{\last_2(1,R)\}$ will be a monochromatic $m$-subset in view of \eqref{enough1s} with $\max B'_1<\max B_1$ and $\diam B'_1\geq \last_2(1,R)-(m-1-\beta-\eta)=2m-2$ (in view of the definition of $\eta$), contradicting the extremal condition (a) for $B_1$.

From \eqref{morew}, we see there are at least $m-1-\beta-\eta$ integers colored by $0$ less than $\first(1,R)$ (with this estimate being rather trivial when $[1,m-1-\beta-\eta]=\emptyset$). By the definition of $\eta$, we have  at least $\eta$ integers colored by $0$ all greater than $\last_2(1,R)$. In particular, since
\eqref{enough1s} implies $\first_m(1,R)\leq \last_2(1,R)$, we find that there are at least $\eta$ integers colored by $0$ greater than $\first_m(1,R)$.
 Since $|\Delta^{-1}(0)\cap R|<m$ holds by the claim hypothesis, this leaves at most $m-1-\eta-(m-1-\beta-\eta)=\beta$ integers that can be colored by $0$ between $\first(1,R)$ and $\first_m(1,R)$, i.e.,
$$|\Delta^{-1}(0)\cap [\first_1(1,R),\first_m(1,R)]|\leq \beta.$$

It now remains to show
$$\beta\geq \alpha$$ and then (iii) will follow, completing the claim. If $\alpha=0$, then this holds trivially, so we now assume $\alpha\geq1,$ in which case  \eqref{waterspout} implies  there are at least $m+1$ integers colored by $1$ in the interval $[m-\alpha-\beta,3m-2-\beta]$.  Recall from \eqref{minB1} that  $\min B_1=m-\alpha-\beta$. Then we must have \be\label{miy}\Delta([m-\alpha-\beta+1,m-\beta])=\{0\},\ee for otherwise $$B'_1=\first_2^{m}(1, [m-\alpha-\beta,3m-2-\beta])\cup \{3m-2-\beta\}$$ will be a monochromatic $m$-subset with $\diam B_1>\diam B'_1\geq 3m-2-\beta-m+\beta=2m-2$, contradicting the extremal condition (b) for $B_1$. But now we have at least $m-1-\beta-\eta$ integers colored by $0$ in $[1,m-1-\beta-\eta]$ by \eqref{morew}, at least  $\alpha$ integers colored by $0$ in $[m-\alpha-\beta+1,m-\beta]$ by \eqref{miy}, and $\eta$ integers colored by $0$ in $[3m-2-\beta-\eta,3m-3-\beta]$ by the definition of $\eta$. Since
$$m-1-\beta-\eta<m-\alpha-\beta+1\leq m-\beta<3m-2-\beta-\eta,$$ where the first inequality follows in view \eqref{wolf1}, the second from the assumption $\alpha\geq 1$ (noted at the start of the paragraph), and the third from the trivial inequality $\eta\leq |\Delta^{-1}(0)\cap R|$ combined with the claim's hypothesis $|\Delta^{-1}(0)\cap R|\leq m-1$, it follows that these three intervals are all disjoint.
Thus $$|\Delta^{-1}(0)\cap R|\geq (m-1-\beta-\eta)+\alpha+\eta=m-1-\beta+\alpha.$$ But combining the above inequality with the claim hypothesis $|\Delta^{-1}(0)\cap R|\leq m-1$ now yields $\beta\geq \alpha$, completing the claim as remarked previously.

\subsection*{Claim C} If $\eta>0$ and $|\Delta^{-1}(0)\cap R|\geq m$, then  \be\label{ladenh}\Delta([1,m-1-\beta])\subseteq \{1\}.\ee Moreover, if we also have  $\beta\leq m-2$ and $|\Delta^{-1}(1)\cap R|> m$, then $\eta\geq m-1-\beta$.

From the first hypothesis $\eta>0$ and the definition of $\eta$, we have  $\Delta(3m-3-\beta)=0$. Thus from the second hypothesis  $|\Delta^{-1}(0)\cap R|\geq m$, we conclude that  \eqref{ladenh} holds,  for otherwise $B'_1=\first_1^{m-1}(0,R)\cup \{3m-3-\beta\}$ will be a monochromatic $m$-subset with $\max B'_1<\max B_1$ and $\diam B'_1\geq 3m-3-\beta-(m-1-\beta)=2m-2$, contradicting the extremal condition (a) for $B_1$.
This completes the first part of the claim, and we now assume the hypotheses of the second part. In view of \eqref{ladenh} and the third hypothesis  $\beta\leq m-2$, it follows that $\Delta(1)=1$, and now we must have $$\Delta([2m-1,3m-3-\beta])=\{0\},$$ for otherwise $B'_1=\first_1^{m-1}(1,R)\cup \{\last_2(1,R)\}$ will be a monochromatic $m$-subset, in view of the fourth hypothesis $|\Delta^{-1}(1)\cap R|>m$, with $\max B'_1<\max B_1$ and $\diam B'_1\geq (2m-1)-1=2m-2$, contradicting the extremal condition (a) for $B_1$. The claim now follows in view of the definition of $\eta$.

\smallskip

In view of Claim B, we may assume \be\label{0aremany}|\Delta^{-1}(0)\cap R|\geq m,\ee else the proof is complete. We divide the remainder of the proof into two cases.

\subsection*{Case 1:}  $\beta=m-1-\alpha$. Note this is equivalent to  $\min B_1=1$ in view of \eqref{minB1}.

If $\alpha=0$, then $\beta=m-1$ and
(ii) follows. So assume $$\alpha\geq 1\quad\und \quad\beta\leq m-2,$$
where the latter inequality follows from the former in view of the case hypothesis. We will show (i) holds.

\smallskip

Suppose $|\Delta^{-1}(1)\cap R|=m$.
Then
$|\Delta^{-1}(0)\cap
R|=2m-2-\beta$. In view of \eqref{0aremany}, we see that, if $|\Delta^{-1}(1)\cap
[\first(0,R),\last(0,R)]|>\beta$, then $B'_1=\first_1^{m-1}(0,R)\cup\{\last_1(0,R)\}$ will be a monochromatic $m$-subset with $\diam(B'_1)\geq |\Delta^{-1}(0)\cap
R|-1+\beta+1= 2m-2$ and $\max B'_1<\max B_1$, contradicting the maximality condition (a) for $B_1$. Therefore we must instead have
$|\Delta^{-1}(1)\cap [\first(0,R),\last(0,R)]|\leq\beta$.
From the definition of $\nu$, there are precisely $\nu+1$ integers colored by $1$ to the right of $\last(0,R)=3m-3-\beta-\nu$ in $R$. Combined with the previous sentence, this means there are at most $\beta+\nu+1$ elements of $R$ colored by $1$ greater than $\first(0,R)$.  But now, since $|\Delta^{-1}(1)\cap R|=m$,  there must be   at least $m-1-\beta-\nu$ integers colored by $1$ to the left of $\first(0,R)$. Consequently, if $m-1-\beta-\nu\geq 1$, then
 (i)
follows by letting $H_1$ be the string given by
$\Delta[\last_{m-1-\beta}(0,R)+1,3m-3-\beta-\nu]$, letting $\mu$ be
the number of $1$'s in $H_1$, and noting that the case hypothesis gives $\beta=m-1-\alpha$ (so that $H_0$ containing exactly $m-1-\alpha-\mu$ $1$'s is equivalent to $|\Delta^{-1}(1)\cap R|=m$). On the other hand, if $m-1-\beta-\nu\leq  0$, then let $\nu'=m-2-\beta< \nu$. Since $\beta\leq m-2$, we also have $\nu'\geq 0$, while $m-1-\beta-\nu'=1$, which is colored by $1$ in view of $\min B_1=1$. Thus (i) follows using $\nu'$ in place of $\nu$ by letting $H_1$ be the string given by
$\Delta[\last_{m-1-\beta}(0,R)+1,3m-3-\beta-\nu']$, and by letting $\mu$ be
the number of $1$'s in $H_1$.
So we may now assume \be\label{1ton}|\Delta^{-1}(1)\cap
R|>m.\ee In particular, \eqref{1ton} and \eqref{0aremany} together force $3m-2-\beta=|R|\geq m+1+m$, implying $\beta\leq m-3$.

\smallskip

In view of \eqref{1ton} and Claim A, it follows that $\eta\geq \alpha>0$, which, together with \eqref{0aremany}, allows us to apply Claim C to conclude $\Delta([1,m-1-\beta])=\{1\}$. Thus, in view of $\beta\leq m-3$, we find that $\Delta(1)=\Delta(2)=1$, i.e., $\first_2(1,R)=2$.
As a result, $B'_1=\first_2^m(1,R)\cup \{3m-2-\beta\}$ is a monochromatic $m$-subset (in view of \eqref{1ton}) with $$\diam B'_1=3m-4-\beta=2m-3+\alpha\geq 2m-2,$$ where the second equality follows by the case hypothesis and the  inequality from the assumption $\alpha\geq 1$. Since $\max B'_1=\max B_1$ and $2m-2\leq \diam B'_1<\diam B_1=2m-2+\alpha$, this contradicts the extremal condition (b) for $B_1$, completing the case.

\subsection*{Case 2:} $\beta<m-1-\alpha$. Note this is equivalent to  $\min B_1>1$ by \eqref{minB1} and implies $$\beta\leq m-2$$ in view of $\alpha\geq 0$.
We divide this case into two subcases.

\subsection*{Subcase 2.1:} $\eta=0$, so that $\last_2(1,R)=3m-3-\beta$.

We will show (ii) holds.
From the subcase hypothesis and  Claim A, it follows that \be\nn\alpha=0\quad\mbox{ or }\quad
|\Delta^{-1}(1)\cap R|=m.\ee From \eqref{minB1}, we know there are $m$ integers in $R$ colored by $1$ all at least $m-\alpha-\beta$, namely, the $m$ integers from $B_1$.
Thus, if $\first(1,R)<m-\alpha-\beta$, then the set $B'_1=
\first_1^{m-1}(1,R)\cup \{\last_2(1,R)\}$ will be a monochromatic $m$-subset with
 $\diam B'_1\geq 3m-3-\beta-(m-\alpha-\beta-1)=2m-2+\alpha\geq 2m-2$ and $\max B'_1<\max B_1$, contradicting the extremal condition (a) for $B_1$.
Therefore we may instead assume $\first(1,R)=\min B_1= m-\alpha-\beta$ (in view of \eqref{minB1}), so that \be\label{otheror}\Delta([1,m-\alpha-\beta-1])=\{0\}\quad\und\quad \Delta(m-\alpha-\beta)=1.\ee
Now $[1,m-\alpha-\beta-1]$ is a nonempty interval (in view of the hypothesis of Case 2) entirely colored by $0$. Consequently,  in view of \eqref{0aremany}, it follows that the extremal condition (a) for $B_1$ will be contradicted by $\first_1^{m-1}(0,R)\cup \{\last(0,R)\}$ unless $\Delta([2m-1,3m-2-\beta])=\{1\}$. However, in this latter case,
(ii) follows (note $H_2$ containing exactly $\beta-1$ $1$'s is equivalent to $|\Delta^{-1}(1)\cap R|=m$, and that this in turn forces $\beta-1\geq 0$), completing the subcase.

\subsection*{Subcase 2.2:} $\eta>0$

We will show (i) holds with $\nu=\mu=0$ in this case. In view of the subcase hypothesis  $\eta>0$ and \eqref{0aremany}, we may apply Claim C to conclude that \be\label{foodil}\Delta([1,m-1-\beta])=\{1\}.\ee
In particular, in view of $\beta\leq m-2$, we see that $\Delta(1)=1$. Thus, we must have
 \be\label{stihy}|\Delta^{-1}(1)\cap R|>m,\ee for otherwise $\min B_1=1$ by \eqref{minB1}, contrary to the hypothesis of Case 2. But now we can further apply Claim C to conclude $$\eta\geq m-1-\beta
\quad\und\quad \Delta([2m-1,3m-3-\beta])=\{0\},$$ where the second statement above is simply a restatement of the first in view of the definition of $\eta$.

If $\alpha=0$, then $\min B_1=m-\beta$ by \eqref{minB1}, in which case  \eqref{foodil} shows that all integers colored by $0$ in $R$ lie between $\min B_1$ and $\max B_1=3m-2-\beta$. Thus, in view of \eqref{0aremany}, it follows that $\diam B_1\geq |B_1|-1\geq 2m-1$, contradicting that $\diam B_1=2m-2+\alpha=2m-2$ by (b). Therefore we instead conclude that $\alpha\geq 1$.

If $\alpha=1$, then $\min B_1=m-\alpha-\beta=m-1-\beta$ by \eqref{minB1}. Thus \eqref{foodil} shows that all integers colored by $0$ in $R$ lie between $\min B_1$ and $\max B_1=3m-2-\beta$. Hence $2m-1=2m-2+\alpha=\diam B_1\geq m-1+|\Delta^{-1}(0)\cap R|$, which combined with \eqref{0aremany} forces there to be exactly $|\Delta^{-1}(0)\cap R|=m$ integers colored by $0$ in $[\min B_1+1,\max B_1-1]=[m-\beta,3m-3-\beta]$, and thus exactly $m-2=m-1-\alpha$ integers colored by $1$ in $[m-\beta,3m-3-\beta]$. As a result, (i) follows with $\mu=\nu=0$. Therefore we may now assume $\alpha\geq 2$.

Since $\alpha\geq 2$, \eqref{foodil} implies that $\Delta(m-\alpha-\beta+1)=1$. As a result, recalling from \eqref{minB1} that  $\min B_1=m-\alpha-\beta$,  we must have $$|\Delta^{-1}(1)\cap [m-\alpha-\beta,3m-2-\beta]|=m,$$ for otherwise  $B'_1=\first_2^m(1,[m-\alpha-\beta,3m-2-\beta])\cup \{3m-2-\beta\}$ will be a monochromatic $m$-subset with $\max B'_1=\max B_1$ and $\diam B'_1=\diam B_1-1=2m-3+\alpha\geq 2m-2$, contradicting the extremal condition (b) for $B_1$.
  But now (i) follows with $\nu=\mu=0$, completing the proof.
\end{proof}

The next lemma translates the structural information from Lemma
2.1 into the existence of sets with small
diameter.

\begin{lem}\label{lem2.2}
Let $m\geq 2$ and let $\Delta: [1,3m-2]\rightarrow\{0,1\}$ be a $2$-coloring.

\begin{itemize}
\item[(i)] If there does not exist a monochromatic $m$-subset $B$ such that
${\diam}(B)\geq 2m-2$, then there exist monochromatic $m$-subsets
$D_1,\,D_2\subseteq [1,3m-2]$ such that $$D_1<_p D_2\quad\und\quad{\diam}(D_1)={\diam}(D_2)=m-1.$$

\item[(ii)] Otherwise, if $\beta$, $\alpha$, $\nu$, $\mu$, and $B_1$ are as
defined in Lemma 2.1, then the following hold.\begin{itemize}

\item[(a)] There exist monochromatic $m$-subsets
$A_1,\,A_2\subseteq [1,3m-2-\alpha-\beta]$ with \ber\nn&&\label{eq1}{\diam}(A_1)\leq
2m-2-\alpha\quad\und\quad\quad{\diam}(A_2)\leq
m+\lfloor\frac{m-1+\beta}{2}\rfloor-1.\eer

\item[(b)] If either Lemma \ref{lem2.1}(iii) holds, or $\alpha\geq 1$ and Lemma \ref{lem2.1}(ii) holds, then $$\diam A_1\leq m-1+\beta.$$
\item[(c)] If Lemma \ref{lem2.1}(i) holds, then  $$\diam A_1\leq 2m-2-\alpha-\mu$$ and there exists a monochromatic $m$-subset $A_3\subseteq [1,m+\alpha+\beta]\subseteq [1,3m-2-\alpha-\beta]$ with $$\diam A_3\leq m+\alpha+\beta-1.$$
\end{itemize}
\end{itemize}
\end{lem}

\begin{proof}
First suppose that there does not exist a monochromatic $m$-subset
$B$ with ${\diam}(B)\geq 2m-2$. We may assume each color is
used at least $m$ times, for otherwise w.l.o.g.
$D=\first_1^{m-1}(1,[1,3m-2])\cup \{\last(1,[1,3m-2])\}$ is a
monochromatic $m$-set with ${\diam}(D)\geq |D|-1\geq 2m-2$, contrary to hypothesis.
We may w.l.o.g. assume $\Delta(1)=1$. Since there does not exist a
monochromatic $m$-subset $B$ with ${\diam}(B)\geq 2m-2$, and since
each color is used at least $m$ times, it follows that
$\last(1,[1,3m-2])\leq 2m-2$, \ $\Delta(3m-2)=0$ and $\first(0, [1,3m-2])\geq
m+1$. Hence $D_1=[1,m]$ and $D_2=[2m-1,3m-2]$ satisfy (i).

So we may now assume there exists a monochromatic $m$-subset $B$ with
${\diam}(B)\geq 2m-2$. Let $\beta$, $\alpha$, $\nu$, $\mu$, and $B_1$
be as defined in Lemma 2.1. Let $R=[1,3m-2-\beta]$ and assume
w.l.o.g. $\Delta(B_1)=\{1\}$.
 Notice that $\beta\leq m-1$ implies
$\beta\leq \lfloor\frac{m-1+\beta}{2}\rfloor$, so that
$$m-1+\beta\leq m+\lfloor\frac{m-1+\beta}{2}\rfloor-1.$$
Applying Lemma 2.1 to $[1, 3m-2]$ yields
three cases.

\subsection*{Case 1:} Lemma \ref{lem2.1}(i) holds.

Suppose $\beta=m-1-\alpha$. Then $H_0$ contains exactly $m-1-\alpha-\mu$ $1$'s and exactly $\alpha+\beta=m-1$ $0$'s.  Thus the string $H_00$ contains $m$ $0$'s and exactly $m-1-\alpha-\mu=\beta-\mu$ $1$'s, in which case   $A_1=A_2=A_3=\first_1^m(0,R)\subseteq [1,2m-1-\mu-\nu]\subseteq [1,m+\alpha+\beta]= [1,3m-2-\alpha-\beta]$ is a monochromatic $m$-subset with $\diam A_1\leq 2m-2-\alpha-\mu=m-1+\beta-\mu$, as desired. So we may now assume $\beta<m-1-\alpha$, in which case \be\label{workit}\mu=\nu=0 \quad\und\quad \Delta[1,3m-2-\alpha-\beta]=1^{m-1-\beta}H_00^{m-\alpha-\beta}.\ee

Since $H_0$ contains exactly $m-1-\alpha$ $1$'s and exactly $\alpha+\beta$ $0$'s, it follows from \eqref{workit} that $A_1=\first_1^m(0,[1,3m-2-\alpha-\beta])$ is a monochromatic $m$-subset such that  $\diam A_1\leq m-1+(m-1-\alpha)=2m-2-\alpha$. Moreover, since $m-1-\beta+m-1-\alpha=2m-2-\alpha-\beta\geq m$ in view of $\beta<m-\alpha-1$, it follows that there are at least $m$ $1$'s in the string $1^{m-1-\beta}H_0$. Consequently, since there are at most $\alpha+\beta$ $0$'s in $H_0$, it follows that $A_3=\first_1^m(1,[1,m+\alpha+\beta])$ is a monochromatic $m$-subset with
 $\diam A_3\leq m+\alpha+\beta-1$ and  $A_3\subseteq [1,m+\alpha+\beta]$. Since $\frac12(\diam A_1+\diam A_3)\leq \frac12(3m-3+\beta)=m+\frac{m-1+\beta}{2}-1$, we can take $A_2$ to be the set from $A_1$ or $A_3$ having smaller diameter, and then $\diam A_2\leq m+\lfloor\frac{m-1+\beta}{2}\rfloor-1$, completing the case.

%
%

\subsection*{Case 2:}  Lemma 2.1(ii) holds.

If $\alpha>0$, then $\beta\leq m-1-\alpha\leq m-2$, and letting $A_1=A_2=\first_1^m(0,R)\subseteq [1,m+\beta]\subseteq [1,3m-2-\alpha-\beta]$, we find that $\diam A_1\leq m-1+\beta\leq 2m-2-\alpha$, as desired. It remain to consider when $\alpha=0$.

Taking $A_1=B_1\subseteq [1,3m-2-\beta]=[1,3m-2-\alpha-\beta]$ gives a monochromatic $m$-subset with $\diam A_1=\diam B_1=2m-2+\alpha=2m-2-\alpha$. If $\beta=m-1$, then $m+\lfloor\frac{m-1+\beta}{2}\rfloor-1=2m-2$, and  we may take $A_2=A_1$.
Therefore assume $\beta<m-1$.
Assume by contradiction that there is no monochromatic $m$-subset
 $A_2\subseteq [1,3m-2-\alpha-\beta]=R$ with $\diam A_2\leq m+\lfloor\frac{m-1+\beta}{2}\rfloor-1$. Since
 $\Delta([2m-1,3m-2-\beta])=\{1\}$, this implies
 \be\label{sly1}|\Delta^{-1}(1)\cap [2m-\lfloor\frac{m-1+\beta}{2}\rfloor-1-\beta, 2m-2]|\leq \beta-1,\ee
 else $A_2=\last_1^m(1,R)$ will be such a set.
 On the other hand, since $\Delta([1,m-1-\beta])=\{0\}$ and $|\Delta^{-1}(0)\cap R|\geq m$, it likewise follows that
 \be\label{sly2}|\Delta^{-1}(0)\cap [m-\beta,m+\lfloor\frac{m-1+\beta}{2}\rfloor]|\leq \beta,\ee else $A_2=\first_1^m(0,R)$ will be such a set.
 Observe that \ber\nn &&2m-\lfloor\frac{m-1+\beta}{2}\rfloor-1-\beta\geq 2m-\lfloor\frac{m-1+m-2}{2}\rfloor-1-\beta=m-\beta+1\quad \und\\ \nn &&m+\lfloor\frac{m-1+\beta}{2}\rfloor\leq m+\lfloor\frac{m-1+m-2}{2}\rfloor =2m-2,\eer both in view of $\beta\leq m-2$. Thus, letting $R'=[2m-\lfloor\frac{m-1+\beta}{2}\rfloor-1-\beta, m+\lfloor\frac{m-1+\beta}{2}\rfloor]$, we see that \eqref{sly1} and \eqref{sly2} yield the contradiction
 $$2\beta\leq 2\lfloor\frac{m-1+\beta}{2}\rfloor-m+2+\beta=|R'|\leq 2\beta-1,$$ completing the case.

\subsection*{Case 3:} Lemma 2.1(iii) holds.

Letting
$\first_1^m(1,[1,3m-2-\beta-\alpha])=A_1=A_2$, it follows that
$\diam (A_1)\leq m-1+\beta\leq 2m-2-\alpha$, completing the proof.
\end{proof}

We are now ready to proceed with our main result.

\begin{thm}Let $m\geq 2$ be an integer. Then $$f(m,m,m;2)= 8m-5+\lfloor\frac{2m-2}{3}\rfloor+\delta,$$ where $\delta=1$ if $m\in \{2,5\}$ and $\delta=0$ otherwise.
\end{thm}

\begin{proof}
Observe that $f(m,m,m;2)\geq 8m-5+\lfloor\frac{2m-2}{3}\rfloor$
follows for $m\geq 2$ by considering the coloring of
$[1,8m-6+\lfloor\frac{2m-2}{3}\rfloor]$ given by the string
$$01^{m-1}0^{m-1}1^{m-1}
0^{\lfloor\frac{2m-2}{3}\rfloor}1^{m-\lfloor\frac{2m-2}{3}\rfloor-1}
0^{m-1}1^{2m-1+\lfloor\frac{2m-2}{3}\rfloor}0^{m-1}.$$ Likewise, $f(2,2,2;2)\geq 12$ follows by considering the coloring of $[1,11]$ given by the string
$$10101101110,$$ and $f(5,5,5;2)\geq 38$
follows by considering the coloring of $[1,37]$ given by the string
$$01^40^41^40^81^40^21^70^3.$$
We proceed
to show that $f(m,m,m;2)\leq 8m-5+\lfloor\frac{2m-2}{3}\rfloor+\delta$, where $\delta=1$ for $m\in \{2,5\}$ and $\delta=0$ otherwise.

Suppose by contradiction that
$\Delta[-2m+2,6m-4+\lfloor\frac{2m-2}{3}\rfloor+\delta]\rightarrow\{0,1\}$
is a $2$-coloring that avoids a monochromatic solution to
$p(m,m,m;2)$ (the problem is translation invariant, so we choose our
interval to begin with $-2m+2$ for notational convenience in the
proof).   From the pigeonhole principle, it follows that there is a
monochromatic $m$-set $$A_0\subseteq [-2m+2,0]\quad\mbox{ with }\quad\diam (A_0)\leq
2m-2.$$ The strategy  is to show that any $2$-coloring of $[1,6m-4+\floor{\frac{2m-2}{3}}+\delta]$ must either contain a monochromatic solution to $p(m,m,m;2)$ or else a monochromatic solution $B<_p C$ to $p(m,m;2)$ with $\diam B\geq 2m-2$. In the latter case, $A_0$, $B$ and $C$ will give the desired monochromatic solution to $p(m,m,m;2)$.

Apply Lemma \ref{lem2.2} to $[1,3m-2]$. If (i) holds, then, since the pigeonhole
principle guarantees there is a monochromatic $m$-set $D_3\subseteq [3m-1,5m-3]$,
we find that  $D_1$, $D_2$ and $D_3$ form a monochromatic solution to
$p(m,m,m;2)$. So we may assume (ii) of Lemma 2.2 holds and we will
apply it and Lemma \ref{lem2.1} to $[1, 3m - 2]$. Let $\alpha$, $\beta$, $\nu$, $\mu$,
$A_1$, $A_2$, $A_3$ and $B_1$ be as in Lemmas \ref{lem2.2} and \ref{lem2.1}.
Let $\Delta(B_1)=\{c_1\}$ with $\{c_1,c_0\}=\{1,0\}$. Thus, when reading the conclusion of Lemma \ref{lem2.1}, we must use $c_1$ in place of $1$ and $c_0$ in place of $0$.
Recall that we have the trivial inequality $$0\leq \alpha+\beta\leq m-1$$ in view of the definition of $\alpha$. Let $$R_1=[1,3m-2-\beta] \quad \und\quad R_2=[3m-1-\beta,6m-4+\lfloor\frac{2m-2}{3}\rfloor+\delta].$$

\subsection*{Step 1:} $|\Delta^{-1}(0)\cap R_2|\geq m$ and $|\Delta^{-1}(1)\cap R_2|\geq m$.

Suppose w.l.o.g. that $|\Delta^{-1}(0)\cap R_2|<m$. Then
$$|\Delta^{-1}(1)\cap R_2|\geq
2m-1+\beta+\lfloor\frac{2m-2}{3}\rfloor+\delta.$$ Let
$$\gamma=|\Delta^{-1}(0)\cap [\first(1,R_2),\last(1,R_2)]|.$$ If
$\alpha\leq \lfloor\frac{2m-2}{3}\rfloor+\gamma+\beta+\delta$, then $A_0$,
$B_1$, and $C=\first_1^{m-1}(1,R_2)\cup \last(1, R_2)$ form a
monochromatic solution to $p(m,m,m;2)$. So we may assume \be\label{seemore}\alpha\geq
\lfloor\frac{2m-2}{3}\rfloor+\gamma+\beta+\delta+1\geq \frac{2m-1}{3}+\delta+\gamma+\beta.\ee Let $y$ be the least
integer such that $\Delta(y)=1$, $y\geq \first(1,R_2)+2m-2-\alpha$ and
$y\geq \first_m(1,R_2)$, and let $$C_1=\first_1^{m-1}(1,R_2)\cup\{y\}.$$ Note $y$ exists in view of $|\Delta^{-1}(1)\cap R_2|\geq 2m-1-\alpha\geq m$.
Observe
that \be\label{davli}|\Delta^{-1}(1)\cap [\first(1,R_2),y]|\leq \max\{2m-1-\alpha,\,m\}=2m-1-\alpha,\ee since
otherwise the minimality of $y$ is contradicted by $y'=\first_{2m-1-\alpha}(1,R_2)$.
Then  \be\label{seemorer}2m-2-\alpha \leq \diam C_1\leq 2m-2-\alpha+\gamma\leq\frac{4m-5}{3}-\delta-\beta,\ee where the latter inequality follows from  \eqref{seemore}.
  By \eqref{davli},  there are at
least
$$2m-1+\beta+\lfloor\frac{2m-2}{3}\rfloor+\delta-(2m-1-\alpha)
=\lfloor\frac{2m-2}{3}\rfloor+\alpha+\beta+\delta$$
integers colored by 1 in
$$R_3=[y+1,6m-4+\lfloor\frac{2m-2}{3}\rfloor+\delta].$$ In view of \eqref{seemore}, we have
$\lfloor\frac{2m-2}{3}\rfloor+\alpha+\beta+\delta\geq m$. Thus, letting
$C_2=\first_1^{m-1}(1,R_3)\cup
\{\last(1,R_3)\}$, it follows from \eqref{seemore} that $$\diam C_2\geq |\Delta^{-1}(1)\cap R_3|-1\geq \frac{2m-4}{3}+\alpha+\beta+\delta-1
\geq \frac{4m-5}{3}+2\delta+2\beta+\gamma-1\geq \frac{4m-5}{3}-1.$$
Comparing the above bound with that of \eqref{seemorer}, we see that if any estimate used in obtaining these bounds can be improved by $1$, then $A_1$, $C_1$ and $C_2$ will be a monochromatic solution to $p(m,m,m;2)$.
Assuming this is not the case, we instead find that $\delta=0$, $\beta=0$, $\gamma=0$,
and  $|\Delta^{-1}(1)\cap R_2|=
2m-1+\beta+\lfloor\frac{2m-2}{3}\rfloor+\delta$.

Since $\delta=0$, we have $$m\geq 3.$$
Since $\delta=\gamma=\beta=0$ and $|\Delta^{-1}(1)\cap R_2|=
2m-1+\beta+\lfloor\frac{2m-2}{3}\rfloor+\delta$, it follows that there are exactly $2m-1+\floor{\frac{2m-2}{3}}$ integers colored by $1$ in $R_2$, all of them consecutive, with the remaining $m-1$ integers colored by $0$. Thus
\be\label{cheesyh}\Delta R_2=0^{\lambda}1^{2m-1+\floor{\frac{2m-2}{3}}}0^{m-1-\lambda},\ee for some $\lambda\in [0,m-1]$.
Since $\alpha=\floor{\frac{2m+1}{3}}>0=\beta$, we cannot have (iii) holding in Lemma \ref{lem2.1}. Since $\alpha>0$ and $\beta=0$, we cannot have (ii) holding in Lemma \ref{lem2.1}. Thus Lemma \ref{lem2.1}(i) must hold.

If $\beta=m-1-\alpha$, then $\alpha=m-1$ in view of $\beta=0$. Moreover, the string  $H_00H_1$ contains exactly $m-1-\alpha=\beta=0$ $c_1$'s. Thus $\Delta R_1=c_1^{m-1-\nu}c_0^{2m-2}c_1^{1+\nu}.$  But now, in view of $m\geq 3$ and \eqref{cheesyh},  it is clear that $[m-\nu,2m-\nu-1]$, $[\first(1,R_2),\first(1,R_2)+m-1]$ and $[\first(1,R_2)+m,\first(1,R_2)+2m-1]$ are three monochromatic sets each of diameter $m-1$, yielding a monochromatic solution to $p(m,m,m;2)$. Therefore we may instead assume $0=\beta<m-1-\alpha$, implying  \be\label{xit}\mu=\nu=\beta=0,\quad\alpha=\floor{\frac{2m+1}{3}}\leq m-2\quad\und\quad m\geq 5,\ee where the final inequality follows from the second.
In this case, Lemma \ref{lem2.1} implies
\be\label{stickygoo}\Delta R_1={c_1}^{m-1}H_0c_0^{m-1}c_1.\ee
We divide the remainder of the proof of Step 1 into two cases.

\subsection*{Case 1.1:}  $c_0=1$.

In this case, we see from \eqref{cheesyh} and \eqref{stickygoo} that
\be\label{oozysauce}\Delta[2m-1,6m-4-\floor{\frac{2m-2}{3}}]=
1^{m-1}0^{\lambda+1}1^{2m-2+\alpha}0^{m-1-\lambda}.\ee  Recall from \eqref{xit} that $$3\leq \alpha=\floor{\frac{2m+1}{3}}\leq m-2$$ and recall from Lemma \ref{lem2.2}(ii)(c) that $A_3\subseteq [1,m+\alpha+\beta]\subseteq [1,2m-2]$, where the second inclusion follows from \eqref{xit},  with $\diam A_3\leq m+\alpha+\beta-1=m+\alpha-1$.

Observe that $m+\lambda+2m-2+\alpha\geq 3m-2+\alpha \geq 2m+2\alpha$ in view of $\alpha\leq m-2$. Consequently, if $\lambda+1\leq \alpha$,  then
it follows from  \eqref{oozysauce} that $C=[2m-1,3m-3]\cup\{3m-2+\alpha\}$ and $D=[3m-1+\alpha, 4m-3+\alpha]\cup \{4m-2+2\alpha\}$ will be monochromatic $m$-subsets with $\diam C=\diam D=m+\alpha-1$, in which case $A_3$, $C$ and $D$ form a monochromatic solution to $p(m,m,m;2)$. Therefore we may instead assume \be\label{sser}\lambda+1\geq\alpha+1.\ee

Observe that $2m-3+\alpha\geq 2m\geq m+\lambda+1$ in view of $\alpha\geq 3$ and $\lambda\leq m-1$. Consequently, it follows from  \eqref{oozysauce} that $C=[2m-1,3m-3]\cup\{3m+\lambda-1\}$ and $D=[3m+\lambda,4m+\lambda-2]\cup\{4m+2\lambda\}$ are monochromatic $m$-subsets with $\diam C=\diam D=m+\lambda\geq m+\alpha$ (in view of \eqref{sser}), in which case $A_3$, $C$ and $D$ form a monochromatic solution to $p(m,m,m;2)$, completing the case.

\subsection*{Case 1.2:} $c_0=0$

In this case, we see from \eqref{cheesyh} and \eqref{stickygoo} that
\be\label{oozysauceII}\Delta[2m-1,6m-4-\floor{\frac{2m-2}{3}}]=
0^{m-1}10^{\lambda}1^{2m-2+\alpha}0^{m-1-\lambda}.\ee  Recall from Lemma \ref{lem2.2}(ii)(c) that $A_3\subseteq [1,m+\alpha+\beta]\subseteq [1,2m-2]$ is a monochromatic $m$-subset, where the second inclusion follows from \eqref{xit}, with $\diam A_3\leq m+\alpha-1$ (in view of $\beta=0$).

If $\lambda\geq \alpha$, then recall that $\alpha\geq 1$ and observe that $2m-2+\alpha\geq m+\alpha$. As a result, we see from \eqref{oozysauceII} that $C=[2m-1,3m-3]\cup\{3m-2+\alpha\}$ and $D=[3m-1+\lambda,4m-3+\lambda]\cup\{4m-2+\lambda+\alpha\}$ are monochromatic $m$-subsets with $\diam C=\diam D= m+\alpha-1$, in which case $A_3$, $C$, and $D$ form a monochromatic solution to $p(m,m,m;2)$.

If $m-\alpha-1\leq \lambda\leq \alpha-1$, then $m-1+\alpha\geq m+\lambda$. As a result, we see from \eqref{oozysauceII} that $C=\{3m-2\}\cup [3m-1+\lambda, 4m-3+\lambda]$ and $D=[4m-2+\lambda,5m-4+\lambda]\cup\{5m-3+2\lambda\}$ are monochromatic $m$-subsets with $\diam C=\diam D=m+\lambda-1\geq 2m-2-\alpha$, in which case $A_1\subseteq [1,3m-2-\alpha]\subseteq [1,3m-3]$, $C$ and $D$ form a monochromatic solution to $p(m,m,m;2)$.

Finally, if $\lambda\leq m-\alpha-1$, then it follows from \eqref{oozysauceII} that $C=\{3m-2\}\cup [4m-2-\alpha, 5m-4-\alpha]$ is a monochromatic $m$-subset with $\diam C=2m-2-\alpha$. Moreover, there are $$2m-2+\alpha-(2m-1-\alpha-1-\lambda)=2\alpha+\lambda\geq 2\alpha\geq 2m-1-\alpha$$ integers greater than $5m-4-\alpha$ that are colored by $1$, where the final inequality follows in view of $3\alpha=3\floor{\frac{2m+1}{3}}\geq 3\frac{2m-1}{3}=2m-1$. Thus, in view of \eqref{oozysauceII}, we conclude  that $D=[5m-3-\alpha,6m-5-\alpha]\cup\{\last(1,R_2)\}$ is a monochromatic $m$-subset with $\diam D\geq 2m-2-\alpha$, in which case $A_1\subseteq [1,3m-2-\alpha]\subseteq [1,3m-3]$, $C$ and $D$ form a monochromatic solution to $p(m,m,m;2)$, completing the case and Step 1.

\bigskip

We may w.l.o.g. assume $\Delta(3m-1-\beta)=1$. Then, since $|\Delta^{-1}(1)\cap R_1|\geq m$ by Step 1,
we must have $\Delta([5m-3-\beta+\alpha, 6m-4+\lfloor\frac{2m-2}{3}\rfloor+\delta])=\{0\}$ else, letting $C=\first_1^{m-1}(1,R_2)\cup \{\last(1,R_2)\}$, it follows that $A_0$,
$B_1$ and $C$ form a monochromatic solution to $p(m,m,m;2)$. But now  we likewise have $\Delta([3m-1-\beta,
4m-2-\alpha+\lfloor\frac{2m-2}{3}\rfloor+\delta])=\{1\}$ else, letting  $C=\first_1^{m-1}(0,R_2)\cup \{\last(0,R_2)\}$, it follows that $A_0$,
$B_1$ and $C$ form a monochromatic solution to $p(m,m,m;2)$.
In summary,
\ber\label{ends-monochromatic-start}
&&\Delta\Big([3m-1-\beta,4m-2-\alpha+\lfloor\frac{2m-2}{3}\rfloor+\delta]\Big)=\{1\}
\quad \und\\ \label{ends-monochromatic-finish}
&&\Delta\Big([5m-3-\beta+\alpha, 6m-4+\lfloor\frac{2m-2}{3}\rfloor+\delta]\Big)=\{0\}.\eer
Note that both of these intervals contain $m-\alpha+\beta+\floor{\frac{2m-2}{3}}+\delta$ integers.

Let \ber\nn C_1&=&[3m-1-\beta,4m-3-\beta]\cup\{4m-2-\alpha+\lfloor\frac{2m-2}{3}\rfloor+\delta\}\quad \und\\ \nn C_2&=&[5m-3-\beta+\alpha,6m-5-\beta+\alpha]\cup \{6m-4+\lfloor\frac{2m-2}{3}\rfloor+\delta\}.\eer When $\beta-\alpha+\floor{\frac{2m-2}{3}}+\delta\geq 0$, these are both monochromatic $m$-subsets of diameter $\diam C_1=\diam C_2=m-1+\beta-\alpha+\floor{\frac{2m-2}{3}}+\delta$.

If $\beta-\alpha+\floor{\frac{2m-2}{3}}+\delta\geq m-1-\alpha$, then $C_1$ and $C_2$ are monochromatic $m$-subsets with $\diam C_1=\diam C_2\geq 2m-2-\alpha$, in which case $A_1$, $C_1$ and $C_2$ form a monochromatic solution to $p(m,m,m;2)$. If $\beta-\alpha+\floor{\frac{2m-2}{3}}+\delta\geq \frac{m-2+\beta}{2}$, then $C_1$ and $C_2$ are monochromatic $m$-subsets with $\diam C_1=\diam C_2\geq m+\ceil{\frac{m-2+\beta}{2}}-1$, in which case $A_2$, $C_1$ and $C_2$ form a monochromatic solution to $p(m,m,m;2)$.
In summary, we may instead assume the contrary of both these inequalities, in turn yielding
\ber\label{1stBasicBound}
&&\beta\leq \frac{m-2}{3}-\delta<m-1\quad\und \\
\label{2ndBasicBound}
&&\alpha\geq \frac{m+1}{6}+\frac{\beta}{2}+\delta>0.
\eer
From \eqref{2ndBasicBound} and  \eqref{1stBasicBound}, we derive that \be \label{starwarz}\alpha>\frac{m-2}{6}+\frac{\beta}{2}=\frac12 (\frac{m-2}{3}+\beta)\geq\beta.\ee
Thus Lemma \ref{lem2.1}(iii) cannot hold.

\smallskip

If $\first(0,R_2)> 5m-2-\beta$, then $\Delta([3m-1-\beta,5m-2-\beta])=\{1\}$, in which case $D_1=[3m-1-\beta,4m-2-\beta]$, $D_2=[4m-1-\beta,5m-2-\beta]$ and $D_3=\first_1^m(0,R_2)$ form a  monochromatic solution to $p(m,m,m;2)$ in view of Step 1. Therefore we may instead assume
\be\label{first0} \first(0,R_2)\leq 5m-2-\beta.\ee

\subsection*{Step 2:} Lemma \ref{lem2.1}(i) holds.

Since Lemma \ref{lem2.1}(iii) does not hold as noted above, assume to the contrary that Lemma \ref{lem2.1}(ii) holds instead. We divide the step into two cases.

\subsection*{Case 2.1:}  $c_1=1$.

In this case, Lemma \ref{lem2.1}(ii) and \eqref{ends-monochromatic-start} yield
$$\Delta\Big([2m-1,4m-2-\alpha+\lfloor\frac{2m-2}{3}\rfloor+\delta]\Big)=\{1\}.$$  Recall that $\alpha+\beta\leq m-1$. Hence, letting
$C=[3m-1-\alpha-\beta, 4m-3-\alpha-\beta]\cup\{4m-2-\alpha\}$, it
follows that $C$ is a monochromatic $m$-set with
$\diam (C)=m-1+\beta$. Consequently, we must have $\first(0,R_2)\geq
5m-2-\beta+\lfloor\frac{2m-2}{3}\rfloor+\delta$, for  otherwise, letting
$D=\first_1^{m-1}(0,
R_2)\cup\{6m-4+\lfloor\frac{2m-2}{3}\rfloor+\delta\}$, it follows in view of Step 1 and Lemma \ref{lem2.2}(ii)(b) that
$A_1\subseteq [1,3m-2-\alpha-\beta]$, $C$ and $D$ form a monochromatic solution to
$p(m,m,m;2)$.
However, since $\lfloor\frac{2m-2}{3}\rfloor+\delta\geq 1$, this is contrary to \eqref{first0}, completing the case.

\subsection*{Case 2.2:}  $c_1=0$.

In this case, we have $\beta\geq 1$ (in view of \eqref{2ndBasicBound} and Lemma \ref{lem2.1}(ii)) and $$\Delta R_1=1^{m-\alpha-\beta-1}0H_20^{m-\beta}$$ with $\Delta[m-\alpha-\beta+1,2m-2]=H_2$ a string of length $m-2+\beta+\alpha$ that contains exactly $\beta-1$ $0$'s and $m-1+\alpha$ $1$'s (in view of $\alpha>0$ from \eqref{2ndBasicBound}). In particular, $\last_{m-1}(1,R_1)\geq m-\beta+1$.

Let $y\leq m-\beta+1$ be the largest integer with $\Delta(y)=1$. Since there are at most $\beta-1$ integers colored by $0$ in $[m-\alpha-\beta+1,m-\beta+1]\subseteq [m-\alpha-\beta+1,2m-2]$ (with the inclusion in view of $\beta\geq 1$), which is an interval of length $\alpha+1\geq \beta+2$ (in view of \eqref{starwarz}), it follows that
\be\label{zoolie}m-2\beta+2\leq y\leq m-\beta+1.\ee
Let
$B=\{y\}\cup \last_1^{m-2}(1,R_1)\cup\{3m-1-\beta\}$. Since $\last_{m-1}(1,R_1)\geq m-\beta+1$, it follows in view of \eqref{zoolie} that $B$ is a monochromatic $m$-subset with  $$2m-2 \leq \diam (B)\leq 2m-3+\beta.$$

If $\first(0,R_2)\leq 4m-1+\floor{\frac{2m-2}{3}}-\beta+\delta$, then $A_0$, $B$ and $\first_1^{m-1}(0,R_2)\cup \{\last(0,R_2)\}$ form a monochromatic solution to $p(m,m,m;2)$ in view of Step 1 and \eqref{ends-monochromatic-finish}. Therefore we may instead assume $$\first(0,R_2)\geq 4m+\floor{\frac{2m-2}{3}}-\beta+\delta>4m-2,$$
with the latter inequality in view of  \eqref{1stBasicBound}.
Hence, letting $C=[3m-1-\beta, 4m-3-\beta]\cup \{4m-2\}$, it follows  that $C$ is a monochromatic $m$-subset with $\diam C= m-1+\beta.$ On the other hand, in view of Step 1, \eqref{first0} and \eqref{ends-monochromatic-finish}, we have $D=\first_1^{m-1}(0,R_2)\cup\{\last(0,R_2)\}$ being a monochromatic $m$-subset with $$\diam D\geq m-2+\beta+\floor{\frac{2m-2}{3}}+\delta\geq m+\beta-1.$$ Thus $A_1$, $C$ and $D$ form a monochromatic solution to $p(m,m,m;2)$ in view of Lemma \ref{lem2.2}(ii)(b) and \eqref{2ndBasicBound}, completing Step 2.

\bigskip

Recall, in view of \eqref{ends-monochromatic-start} and \eqref{ends-monochromatic-finish}, that
$C_1$ and
$C_2$ are both monochromatic $m$-subsets of diameter $\diam C_1=\diam C_2=m-1+\beta-\alpha+\floor{\frac{2m-2}{3}}+\delta$ when $\beta-\alpha+\floor{\frac{2m-2}{3}}+\delta\geq 0$.

If $\beta-\alpha+\floor{\frac{2m-2}{3}}+\delta\geq m-1-\alpha-\mu$, then $C_1$ and $C_2$ are monochromatic $m$-subsets with $\diam C_1=\diam C_2\geq 2m-2-\alpha-\mu$, in which case $A_1$, $C_1$ and $C_2$ form a monochromatic solution to $p(m,m,m;2)$ in view of Lemma \ref{lem2.2}(ii)(c) and Step 2. Likewise, if $\beta-\alpha+\floor{\frac{2m-2}{3}}+\delta\geq \alpha+\beta$, then $C_1$ and $C_2$ are monochromatic $m$-subsets with $\diam C_1=\diam C_2\geq m+\alpha+\beta-1$, in which case $A_3$, $C_1$ and $C_2$ form a monochromatic solution to $p(m,m,m;2)$.
In summary, we may instead assume the contrary of both these inequalities, in turn yielding
\ber\label{3rdBasicBound}
&&\beta\leq \frac{m-2}{3}-\mu-\delta\quad\und \\
\label{4thBasicBound}
&&\alpha\geq \frac{m+\delta}{3}>0.
\eer
In particular, \eqref{3rdBasicBound} and $\beta\geq 0$ yield $$m\geq 3.$$



\subsection*{Step 3:} $c_1=1$.

Assume by contradiction that $c_1=0$. We divide the step into two cases.

\subsection*{Case 3.1:} $\beta=m-1-\alpha$.

In this case, Lemma \ref{lem2.1}(i) yields
$$\Delta[1,3m-2-\beta]=0^{m-1-\beta-\nu}H_01H_10^{1+\nu}$$  with the string
$\Delta[m-\beta-\nu, 2m-2-\nu-\mu]=H_0$ containing exactly $\beta-\mu=m-1-\alpha-\mu$ $0$'s.
We trivially have $m-1-\beta-\nu\geq 0$ and $m-1-\alpha-\mu=\beta-\mu\geq 0$ as these quantities from Lemma \ref{lem2.1}(i) must be nonzero, implying $\nu+\mu\leq m-1$. Thus $2m-2-\nu-\mu\geq m-1$, which means the string $H_01$ covers the interval $[m-\beta-\nu,m]$, thus ensuring  there are at most $\beta-\mu$ integers colored by $0$ in $[m-\beta-\nu,m]$. Since this interval contains at least $\beta+1$ elements, this ensures that there is some  \be\label{y-exists}y\in [m-\beta+\mu,m]\quad \mbox{ with } \quad \Delta(y)=1 \und y\leq \last_{m-1-\beta}(1,R_1),\ee where the later inequality follows by recalling from Lemma \ref{lem2.1}(i) that the string $H_1$ contains exactly $m-2-\beta$ $1$'s. From \eqref{ends-monochromatic-start} and $\alpha=m-1-\beta$, we know $$\Delta([3m-1-\beta,3m-1+\beta+\floor{\frac{2m-2}{3}}+\delta])=\{1\}.$$
Thus $B=\{y\}\cup \last_1^{m-2-\beta}(1,R_1)\cup [3m-1-\beta,3m-1]$ is a monochromatic $m$-subset (in view of \eqref{y-exists}) with $$2m-1\leq \diam B\leq 2m-1+\beta.$$

If $\first(0,R_2)\leq 4m-3-\beta+\floor{\frac{2m-2}{3}}+\delta$, then $A_0$, $B$ and $\first_{1}^{m-1}\cup\{\last(0,R_2)\}$ form a monochromatic solution to $p(m,m,m;2)$ in view of Step 1 and \eqref{ends-monochromatic-finish}. Therefore we may instead assume $$\first(0,R_2)\geq 4m-2-\beta+\floor{\frac{2m-2}{3}}+\delta>4m-2,$$ with the latter inequality following from \eqref{3rdBasicBound}. Thus  $C=[3m-1-\beta,4m-3-\beta]\cup \{4m-2\}$ is a monochromatic $m$-subset with $\diam C=m-1+\beta$. On the other hand, in view of Step 1, \eqref{ends-monochromatic-finish} and \eqref{first0}, we see that $D=\first_{1}^{m-1}(0,R_2)\cup \{\last(0,R_2)\}$ is a monochromatic $m$-subset with $\diam D\geq m-2+\beta+\floor{\frac{2m-2}{3}}+\delta\geq m-1+\beta$, whence $A_1$, $C$ and $D$ give a monochromatic solution to $p(m,m,m;2)$ in view of Lemma \ref{lem2.2}(ii) and $\alpha=m-1-\beta$, completing the case.

\subsection*{Case 3.2:} $\beta<m-1-\alpha$.

In this case, \be\label{cricket}\alpha\leq m-2-\beta\leq m-2,\ee while Lemma \ref{lem2.1}(i) yields $\mu=\nu=0$ and
\be\label{singingl}\Delta[1,3m-2-\beta]=0^{m-1-\beta}H_01^{m-1-\beta}0\ee  with the string
$\Delta[m-\beta, 2m-2]=H_0$ containing exactly $m-1-\alpha$ $0$'s and exactly $\alpha+\beta$ $1$'s.
  Now $m-1-\beta\geq \alpha+1$ by case hypothesis, and $4m-2-\alpha+\floor{\frac{2m-2}{3}}\geq 5m-3-2\alpha$ in view of \eqref{4thBasicBound}.
Thus, if $\alpha\geq 2$, then  \eqref{singingl}, \eqref{cricket} and \eqref{ends-monochromatic-start} imply  $B=[3m-1-\alpha-\beta,3m-3-\beta]\cup [4m-3-\alpha-\beta,5m-3-2\alpha-\beta]$ is a monochromatic $m$-subset with $\diam B=2m-2-\alpha$ and $B\subseteq [3m-1-\alpha-\beta,\first(0,R_2)-1]$. On the other hand, if $\alpha\leq 1$, then \eqref{4thBasicBound} ensures that $\alpha=1$, $m=3$ (recall that we now know $m\geq 3$), $\beta=0$ (in view of the case hypothesis $\beta\leq m-2-\alpha$), and $\diam A_3=m+\alpha+\beta-1=2m-2-\alpha$ (by Lemma \ref{lem2.2}(ii)). In this case, $B=\{3m-3-\beta\}\cup [4m-3-\alpha-\beta,5m-5-\alpha-\beta]$ is a monochromatic $m$-subset with $\diam B=2m-2-\alpha$ and $\min B\geq m+\alpha+\beta+1$.

In view of Step 1, \eqref{ends-monochromatic-finish} and \eqref{first0}, it follows that $C=\first_{1}^{m-1}(0,R_2)\cup \{\last(0,R_2)\}$ is a monochromatic $m$-subset with \be\label{nodstoim}\diam C\geq m-2+\beta+\floor{\frac{2m-2}{3}}+\delta\geq m-2+\floor{\frac{2m-2}{3}}\geq 2m-3-\alpha,\ee with the latter inequality once more in view of \eqref{4thBasicBound}. Thus $A_1$ (if $\alpha\geq 2$) or $A_3$ (if $\alpha=1$), $B$ and $C$ will form a monochromatic solution to $p(m,m,m;2)$ unless equality holds in all the estimates used to derive \eqref{nodstoim}. In particular, we must have $\delta=0$, $\beta=0$, $\alpha=\floor{\frac{m+1}{3}}$ and $\first(0,R_2)=5m-2-\beta=5m-2$.

 Since $\beta=0$, \eqref{singingl} gives  $$\Delta R_1=0^{m-1}H_01^{m-1}0.$$
In view of $\first(0,R_2)=5m-2$ and $\beta=0$, we have $$\Delta([3m-1,5m-3])=\{1\}.$$ Thus, since $\alpha\leq \frac{m+1}{3}$ and $m\geq 3$ imply $4m-2+2\alpha\leq 4m-2+\frac{2m+2}{3}< 5m-2$, it follows that
$D_1=[2m-1,3m-3]\cup \{3m-2+\alpha\}$ and $D_2=[3m-1+\alpha,4m-3+\alpha]\cup \{4m-2+2\alpha\}$ are monochromatic $m$-subsets with $\diam D_1=\diam D_2=m-1+\alpha$.  So $A_3\subseteq [1,m+\alpha+\beta]\subseteq [1,2m-2]$ (the inclusion follows from the case hypothesis), $D_1$ and $D_2$ form a monochromatic solution to $p(m,m,m;2)$ in view of Lemma \ref{lem2.2}(ii)(c) and $\beta=0$, completing the case and Step 3.

\bigskip

In view of Steps 2 and 3 and Lemma \ref{lem2.1}(i), we now have
\ber\label{forgetton}
\Delta[1,3m-2-\beta]=1^{m-1-\beta-\nu}H_00H_11^{1+\nu},
\eer
where $H_0=\Delta[m-\beta-\nu, 2m-2-\mu-\nu]$  is a string containing exactly $m-1-\alpha-\mu$ $1$'s and $\alpha+\beta$ $0$'s and $H_1=\Delta[2m-\mu-\nu,3m-3-\beta-\nu]$ is a string containing exactly $\mu$ $1$'s and $m-2-\beta$ $0$'s. Let $$R'_2=[3m-2-\beta-\nu,6m-4+\lfloor\frac{2m-2}{3}\rfloor]$$
and observe that $\Delta([3m-2-\beta-\nu,4m-2-\alpha+\floor{\frac{2m-2}{3}}+\delta])=\{1\}$ in view of \eqref{ends-monochromatic-start} and \eqref{forgetton}.

If $\first(0,R_2)> 5m-3-\beta-\nu$, then $\Delta([3m-2-\beta-\nu,5m-3-\beta-\nu])=\{1\}$, in which case $D_1=[3m-2-\beta-\nu,4m-3-\beta-\nu]$, $D_2=[4m-2-\beta-\nu,5m-3-\beta-\nu]$ and $D_3=\first_1^m(0,R_2)$ form a  monochromatic solution to $p(m,m,m;2)$ in view of Step 1. Therefore we may instead assume
\be\label{first0-better} \first(0,R_2)\leq 5m-3-\beta-\nu.\ee

Let \be
\nn\gamma_0=|\Delta^{-1}(0)\cap [\first(1,R_2),\last(1,R_2)]|.\ee
We must have \be\label{1stKing} \gamma_0 \leq |\Delta^{-1}(0)\cap R_2|+\alpha-m-\floor{\frac{2m-2}{3}}-\beta\leq |\Delta^{-1}(0)\cap R_2|+\alpha-\frac{5m-4}{3}-\beta,\ee for otherwise $C=\first_{1}^{m-1}(1,R_2)\cup \{\last(1,R_2)\}$ will be a monochromatic $m$-subset (in view of Step 1) with \ber\nn\diam(C)&=&|\Delta^{-1}(1)\cap R_2|+\gamma_0-1=(3m-2+\beta+\floor{\frac{2m-2}{3}}+\delta-|\Delta^{-1}(0)\cap R_2|)+\gamma_0-1
\\\nn &\geq& 2m-2+\alpha,\eer in which case $A_0$, $B_1$ and $C$ form a monochromatic solution to $p(m,m,m;2)$.

\subsection*{Step 4:} $\nu\leq \alpha-1$ and $\alpha\geq \frac{2m+2}{3}+\mu+\nu+\delta$.

To simplify notation, we proceed in two cases.

\subsection*{Case 4.1:} $\beta=m-\alpha+1$.

In this case, Lemma \ref{lem2.2}(ii)(c) shows $\diam A_1\leq m+\beta-\mu-1$ with $A_1\subseteq [1,3m-2-\alpha-\beta]$.
Suppose $\nu\geq \alpha-1-\lfloor\frac{2m-2}{3}\rfloor-\mu-\delta$. Then \eqref{forgetton} and \eqref{ends-monochromatic-start} imply that the interval
$$I=[3m-1-\alpha-\beta+\mu+\lfloor\frac{2m-2}{3}\rfloor+\delta,
4m-2-\alpha+\lfloor\frac{2m-2}{3}\rfloor+\delta]$$ is entirely colored by $1$. Since $|I|\geq m+\beta-\mu$, it follows that $B=\first_1^{m-1}(1,I)\cup \{\first_{m+\beta-\mu}(1,I)\}$ is a monochromatic $m$-subset  with $\diam B=m+\beta-\mu-1$ and $\min B\geq 3m-1-\alpha-\beta$. Thus, in view of \eqref{first0-better}, \eqref{ends-monochromatic-finish} and Step 1, we see that $A_1$, $B$ and $\first_{}^{m-1}(0,R_2)\cup \{\last(0,R_2)\}$ form a monochromatic solution to $p(m,m,m;2)$. So we may instead assume $$\nu\leq \alpha-2-\lfloor\frac{2m-2}{3}\rfloor-\mu-\delta\leq \alpha-2,$$ which rearranges to yield the desired bound for $\alpha$.

\subsection*{Case 4.2:} $\beta<m-\alpha+1$.

In this case, Lemma \ref{lem2.1}(i) implies that $\mu=\nu=0$, so that $\nu\leq \alpha-1$ follows by \eqref{4thBasicBound}. In particular, $\min R'_2\geq 3m-1-\alpha-\beta$. Assume by contradiction that
\be\label{alpha-small}\alpha\leq \frac{2m+1}{3}+\delta.\ee Then
$4m-3\leq 4m-2-\alpha+\lfloor\frac{2m-2}{3}\rfloor+\delta$, so that  \eqref{forgetton} and \eqref{ends-monochromatic-start} ensure
\be\label{soothtooth}\Delta([3m-2-\beta,
4m-3])=\{1\},\ee which is an interval of length $m+\beta\geq m$.
Let $y=\last_2(1,R_1)$. Then $y\leq 2m-2$ by \eqref{forgetton}. If $y<2m-2-\beta$, then $D_1=[2m-2-\beta,3m-3-\beta]$ is a monochromatic $m$-subset (by \eqref{forgetton}) with $\diam D_1=m-1$, in which case $D_1$, $\first_1^m(1,R'_2)$ and $\first_1^m(0,R'_2)$ form a monochromatic solution to $p(m,m,m;2)$ in view of Step 1 and \eqref{soothtooth}. Therefore, we instead have $2m-2-\beta\leq y\leq 2m-2$, which means $D_2=\{y\}\cup [3m-1,4m-3]$ is a monochromatic $m$-subset (by \eqref{soothtooth}) with $$2m-1\leq \diam D_2 \leq 2m-1+\beta.$$ Thus, in view of Step 1 and \eqref{ends-monochromatic-finish}, we must have
$$\first(0,R_2)\geq 4m-2-\beta+\floor{\frac{2m-2}{3}}+\delta,$$ else $A_0$, $D_2$ and $\first_1^{m-1}(0,R_2)\cup \{\last(0,R_2)\}$ form a monochromatic solution to $p(m,m,m;2)$. Consequently, since $5m-4-\beta-\alpha\leq 4m-3-\beta+\floor{\frac{2m-2}{3}}+\delta$ (by \eqref{4thBasicBound}), we see that $D_3=[3m-2-\beta,4m-4-\beta]\cup\{5m-4-\beta-\alpha\}$ is a monochromatic $m$-subset with $\diam D_3=2m-2-\alpha$.
Moreover, in view of Step 1, \eqref{first0-better} and \eqref{ends-monochromatic-finish}, we see that $D_4=\first_1^{m-1}(0,R_2)\cup \{\last(0,R_2)\}$ is a monochromatic $m$-subset with $$\diam D_4\geq m-1+\beta+\frac{2m-4}{3}+\delta\geq m-1+\frac{2m-4}{3}>2m-3-\alpha,$$ where the latter inequality follows in view of \eqref{4thBasicBound}. Thus $A_1$, $D_3$ and $D_4$ form a monochromatic solution to $p(m,m,m;2)$ (in view of Lemma \ref{lem2.2}(ii) and $\alpha\geq 1$), completing the case and Step 4.


\bigskip

In view of Step 4 and the basic inequality $\alpha+\beta\leq m-1$, we have
$$m\geq 6 \quad \und \quad \delta=0,$$
$R'_2\subseteq [3m-1-\alpha-\beta,6m-4+\floor{\frac{2m-2}{3}}]$,
\be\label{KingUp-clone}\alpha\geq \frac{2m+2}{3}+\mu+\nu\quad\und \quad \beta\leq\frac{m-5}{3}-\mu-\nu.\ee

\subsection*{Step 5:} $\last(1,R_2)\geq 5m-2-\alpha-\beta$.

Assume to the contrary that $\last(1,R_2)\leq 5m-3-\alpha-\beta$. Then $$I=[5m-2-\alpha-\beta,6m-4+\floor{\frac{2m-2}{3}}]$$ is an interval entirely colored by $0$ with \be\label{Ibig}|I|= m-1+\alpha+\beta+\floor{\frac{2m-2}{3}}.\ee
In particular, $|I|\geq  m-1+\frac{2m+2}{3}+\beta+\frac{2m-4}{3}\geq 2m+\beta+\frac{m-5}{3}\geq 2m+2\beta$ by \eqref{KingUp-clone}.
Consequently, if $\beta=m-1-\alpha$, then $A_1$, $\first_1^{m-1}(0,I)\cup \{\first_{m+\beta}(0,I)\}$ and $\first_{m+\beta+1}^{2m+\beta-1}(0,I)\cup\{\first_{2m+2\beta}(0,I)\}$ will form a monochromatic solution to $p(m,m,m;2)$ in view of Lemma \ref{lem2.2}(ii)(c). Therefore we may instead assume $\beta<m-1-\alpha$, in which case Lemma \ref{lem2.1}(i) implies $\mu=\nu=0$.

In view of $\beta< m-1-\alpha$, we have $\min I\geq 4m$.
Thus we must have \be\label{contrapos}|\Delta^{-1}(1)\cap [3m-2-\beta,4m-3]|\leq m-1,\ee for otherwise $\first_1^m(1,R'_2)$, $\first_1^{m-1}(0,I)\cup \{\first_{m+\beta}(0,I)\}$ and $\first_{m+\beta+1}^{2m+\beta-1}(0,I)\cup\{\first_{2m+2\beta}(0,I)\}$ will form a monochromatic solution to $p(m,m,m;2)$.
But this means \be\label{wormteel}\first_{\beta+1}(0,R_2)\leq 4m-3.\ee

By \eqref{forgetton} and $\mu=0$, we know $\Delta[2m-1,3m-3-\beta]=0^{m-1-\beta}$.
Thus $$C=\first_1^{m-1}(0,[2m-1,6m-4+\lfloor{\frac{2m-2}{3}}])\cup \{5m-2-\alpha-\beta\}$$ is a monochromatic $m$-subset (in view of \eqref{wormteel}) with $\diam C=3m-3-\alpha-\beta\geq 2m-1$ (in view of $\beta<m-1-\alpha$). On the other hand, $$D=[5m-1-\alpha-\beta,6m-3-\alpha-\beta]\cup\{6m-4+\floor{\frac{2m-2}{3}}\}$$ is a monochromatic $m$-subset (in view of \eqref{Ibig})  with $\diam D=m-3+\alpha+\beta+\floor{\frac{2m-2}{3}}$. However, in view of $\beta\geq 0$ and \eqref{KingUp-clone}, we have $m-3+\alpha+\beta+\floor{\frac{2m-2}{3}}\geq 3m-3-\alpha-\beta$. Hence  $A_0$, $C$ and $D$ form a monochromatic solution to $p(m,m,m;2)$, completing Step 5.

\bigskip

Let $\beta'=m-1-\alpha\geq \beta$. In view of \eqref{KingUp-clone}, we have
\be\label{final-BasicBound} \beta\leq \beta'\leq \frac{m-5}{3}-\mu-\nu.\ee
Let
$y$ be the least integer such that $\Delta(y)=1$,  $y\geq \first_m(1,R'_2)$ and $$y\geq
\first(1,R'_2)+m+\beta'-\mu-1=5m-4-\alpha-\beta-\mu-\nu.$$ Note $y$ exists in view of Steps 1 and 5. Let $$R_3=[3m-2-\beta-\nu,y]$$ and
let $$\gamma'_0=|\Delta^{-1}(0)\cap R_3|\quad\und\quad \gamma''_0=|\Delta^{-1}(0)\cap [\first(1,R_3),\last_2(1,R_3)]|.$$ Note that
$\first(1,R_3)=\min R_3=\min R'_2=3m-2-\beta-\nu$.


If $|\Delta^{-1}(1)\cap R_3|>m$, then we must have $|\Delta^{-1}(1)\cap R_3|-1+\gamma''_0\leq m+\beta'-\mu-1$, else the minimality of $y$ will be contradicted by $\last_2(1,R_3)$. Thus \be\label{merwave-clone} |\Delta^{-1}(1)\cap R_3|\leq m+\max\{0,\beta'-\mu-\gamma''_0\}\leq m+\beta'-\mu.\ee
Consequently, if $\last(0,R_3)-\first(0,R_3)\geq m-2+\max\{\gamma'_0,\beta'-\mu\}$, then
$[\first(0,R_3),\last(0,R_3)]$ is an interval of length at least $m-1+\gamma'_0$ that, by definition of $\gamma'_0$, can contain at most $\gamma'_0$ integers colored by $0$ and, consequently, must contain at least $m-1$ integers colored by $1$. Since $y>\last(0,R_3)$ and is also colored by $1$, this  would mean there are no more than $|\Delta^{-1}(1)\cap R_3|-m\leq \beta'-\mu\leq m-1-\alpha$ (with the second inequality by \eqref{merwave-clone}) integers colored by $1$ in $[\min R_3,\first(0,R_3)-1]=[3m-2-\beta-\nu,\first(0,R_3)-1]$. However, by \eqref{ends-monochromatic-start} and \eqref{forgetton}, we know the first $m-\alpha+\beta+\floor{\frac{2m-2}{3}}+1+\nu\geq m-\alpha$ consecutive integers of $R_3$ are colored by $1$, making this impossible. Therefore we instead conclude that $\last(0,R_3)-\first(0,R_3)\leq m-3+\max\{\gamma'_0,\beta'-\mu\}$, in which case it is easily seen that there is a monochromatic $m$-subset $C_0\subseteq [3m-2-\beta-\nu,y]\subseteq [3m-1-\alpha-\beta,y]$ (where the second inclusion follows from Step 4) with \be\label{C0Diam-clone}\diam C_0=m-1+\max\{\gamma'_0,\beta'-\mu\}\ee (since $\diam R_3\geq m-1+\max\{\gamma'_0,\beta'-\mu\}$ in view of the definitions of $R_3$, $y$ and $\gamma'_0$).

\subsection*{Step 6:} $|\Delta^{-1}(0)\cap R_2|\geq m+\beta+1$.

Suppose by contradiction that $|\Delta^{-1}(0)\cap R_2|\leq m+\beta$, so that \be\label{empress-clone}|\Delta^{-1}(1)\cap R'_2|\geq (3m-1+\beta+\frac{2m-4}{3}+\nu)-m-\beta=\frac{8m-7}{3}+\nu.\ee
Then it follows from \eqref{merwave-clone} and \eqref{empress-clone} that there are at least $\frac{8m-7}{3}-m-\beta'+\mu+\nu=\frac{5m-7}{3}-\beta'+\mu+\nu$ integers colored by $1$ that are greater than $y$. Thus we must have \be\label{whirlit-clone}\frac{5m-7}{3}-\beta'+\mu+\nu\leq m-1+\max\{\gamma'_0,\beta'-\mu\},\ee for otherwise $A_1$, $C_0$ and $\first_{1}^{m-1}(1,R_2\setminus R_3)\cup \{\last(1,R_2)\}$ will be a monochromatic solution to $p(m,m,m;2)$.

If $\gamma'_0\leq \beta'-\mu$, then \eqref{whirlit-clone} implies $\beta'\geq \frac{m-2}{3}+\mu$, contrary to \eqref{final-BasicBound}. On the other hand, if  $\gamma'_0\geq \beta'-\mu+1$, then \eqref{whirlit-clone} instead yields $$\frac{2m-4}{3}-\beta'+\mu+\nu\leq \gamma'_0\leq \gamma_0\leq \alpha-\frac{2m-4}{3},$$ where the final inequality follows from \eqref{1stKing} and the assumption $|\Delta^{-1}(0)\cap R_2|\leq m+\beta$.
Rearranging the above inequality and applying the estimate \eqref{final-BasicBound}, we find that $\alpha\geq m-1+2\mu+2\nu$. Since we trivially have $\beta+\alpha\leq m-1$, we conclude that $\alpha=m-1$ and $\beta=\mu=\nu=0$  with equality holding in all estimates used to derive the bound $\alpha\geq m-1+2\mu+2\nu$. In particular, $0=m-1-\alpha=\beta'=\frac{m-5}{3}-\mu-\nu=\frac{m-5}{3}$, contradicting that we now have $m\geq 6$. This completes Step 6.


\subsection*{Step 7:} $\gamma'_0\geq \beta+2$.

Suppose by contradiction that $\gamma'_0\leq \beta+1\leq \beta'+1$. Then $$m-1+\beta'-\mu\leq \diam C_0\leq m+\beta'$$ by \eqref{C0Diam-clone}.
 Also, in view of Step 6, there are at least $m$ integers colored by $0$ greater than $y$.  As a result,   $$\first(0,R_2\setminus R_3)\geq
5m-3+\lfloor\frac{2m-2}{3}\rfloor-\beta',$$ for  otherwise $A_1$,
$C_0$ and $\first_1^{m-1}(0,
R_2\setminus R_3)\cup\{6m-4+\lfloor\frac{2m-2}{3}\rfloor\}$
form a monochromatic solution to $p(m,m,m;2)$ in view of \eqref{ends-monochromatic-finish} and Lemma \ref{lem2.2}(ii).
Thus  \be\label{pipedragen-clone}\Delta([y,5m-4+\lfloor\frac{2m-2}{3}\rfloor-\beta'])=\{1\}.\ee
We handle two cases.

\subsection*{Case 7.1:} $|\Delta^{-1}(1)\cap R_3|=m$.

In this case, the assumption $\gamma'_0\leq \beta+1$ yields $$y=\min R_2+|\Delta^{-1}(1)\cap R_3|+\gamma'_0-1\leq 4m-2-\nu,$$  while \eqref{final-BasicBound} implies $5m-3+\beta'\leq 5m-4+\lfloor\frac{2m-2}{3}\rfloor-\beta'$.
Thus $$D=[4m-2,5m-4]\cup\{5m-3+\beta'\}$$ is a monochromatic $m$-subset by \eqref{pipedragen-clone} with $\diam D=m-1+\beta'$. Consequently, we must now have
 $$\first(0,R_2\setminus R_3)\geq
5m-2+\lfloor\frac{2m-2}{3}\rfloor-\beta',$$ for  otherwise $A_1$,
$D$ and $\first_1^{m-1}(0,
R_2\setminus R_3)\cup\{6m-4+\lfloor\frac{2m-2}{3}\rfloor\}$
form a monochromatic solution to $p(m,m,m;2)$ in view of Step 6, \eqref{ends-monochromatic-finish},  and Lemma \ref{lem2.2}(ii).

In consequence, $$D'=[4m-1,5m-3]\cup\{5m-3+\lfloor\frac{2m-2}{3}\rfloor-\beta'\}$$ is a monochromatic $m$-subset with $\min D'>y$ and  $\diam D'\geq m-2+\frac{2m-4}{3}-\beta'> m-1+\beta'$ in view of \eqref{final-BasicBound}. Thus $A_1$, $C_0$ and $D'$ form a monochromatic solution to $p(m,m,m;2)$ in view of Lemma \ref{lem2.2}(ii), completing the case.

\subsection*{Case 7.2:} $|\Delta^{-1}(1)\cap R_3|>m$.

In this case,  \eqref{merwave-clone} and the assumption $\gamma'_0\leq \beta+1$ together imply that \ber\nn y&=&\min R_2+|\Delta^{-1}(1)\cap R_3|+\gamma'_0-1\\\nn &\leq& 3m-2-\nu+|\Delta^{-1}(1)\cap R_3|\leq 4m-2+\beta'-\gamma''_0-\mu-\nu\eer and that $B_0=\first_1^m(1,R_3)$ is a monochromatic $m$-subset with $$\diam B_0\leq m+\gamma''_0-1\quad\und\quad \max B_0<y.$$
Consequently, since $5m-3+\beta'\leq  5m-4+\lfloor\frac{2m-2}{3}\rfloor-\beta'$ by \eqref{final-BasicBound}, it follows from \eqref{pipedragen-clone} that  $$C=[4m-2+\beta'-\gamma''_0, 5m-4+\beta'-\gamma''_0]\cup \{5m-3+\beta'\}$$ is a monochromatic $m$-subset  with $\min C\geq y>\max B_0$ and $\diam C=m+\gamma''_0-1$. Thus we must have $\first(0,R_2\setminus R_3)\geq 5m-2+\floor{\frac{2m-2}{3}}-\gamma''_0$ else $B_0$, $C$ and $\first_1^{m-1}(0,R_2\setminus R_3)\cup\{6m-4+\floor{\frac{2m-2}{3}}\}$ will be a monochromatic solution to $p(m,m,m;2)$. But that means $$D=[4m-1+\beta'-\gamma''_0, 5m-3+\beta-\gamma''_0]\cup
\{5m-3+\floor{\frac{2m-2}{3}}-\gamma''_0\}$$ is a monochromatic $m$-subset with $\min D>y$ and $\diam D\geq m-2+\frac{2m-4}{3}-\beta'> m+\beta' -1$ (in view of \eqref{final-BasicBound}). Hence $A_1$, $C_0$ and $D$ form a monochromatic solution to $p(m,m,m;2)$ in view of Lemma \ref{lem2.2}(ii), completing the case and Step 7.

\bigskip


Let $z=\first_{\beta+1}(0,R_2)$ and let $z'=\first_{\beta+2}(0,R_2)$. In view of Step 6, there are at least $m$ integers colored by $0$ greater than $z$. In view of Step 7, we have $z,\,z'\in R_3$ with $z<z'<y$. Thus, in view of the definition of $z'$ and  \eqref{merwave-clone}, we have \ber\nn z'&\leq&
\min R_3 +|\Delta^{-1}(1)\cap R_3\setminus \{y\}|+|\Delta^{-1}(0)\cap [\min R_3,z']|-1\\
\nn &=& \min R_3 +|\Delta^{-1}(1)\cap R_3\setminus \{y\}|+|\Delta^{-1}(0)\cap [\min R_2,z']|-1
\\\nn&=&
\min R_3 +|\Delta^{-1}(1)\cap R_3|+\beta\\\nn&\leq&  4m-2-\nu+\beta'-\mu.\eer
Thus $D=\{z'\}\cup \last_1^{m-1}(0,R_2)$ is a monochromatic $m$-subset with $\min D=z'>z$ and \be\label{ushy-clone}\diam D\geq 2m-2+\frac{2m-4}{3}-\beta'+\mu+\nu\geq 2m-2+\frac{m+1}{3}+2\mu+2\nu,\ee where the second inequality follows from \eqref{final-BasicBound}.

Recall from \eqref{forgetton} and the definition of $z$ that $\Delta(2m-1-\mu-\nu)=0$ with the interval $[2m-1-\mu-\nu,z]$ containing exactly $m$ integers colored by $0$ and at most $$\mu+|\Delta^{-1}(1)\cap R_3\setminus \{y\}|\leq\mu+m+\beta'-\mu-1=m+\beta'-1$$ integers colored by $1$ (in view of \eqref{merwave-clone}). Thus, if there are at least $m-1$ integers colored by $1$ in $[2m-1-\mu-\nu,z]$, then $C=\first_1^m(0,[2m-1-\mu-\nu,z])$ will be a monochromatic $m$-subset with $2m-2\leq C\leq 2m-2+\beta'$, in which case $A_0$, $C$ and $D$ will form a monochromatic solution to $p(m,m,m;2)$ in view of \eqref{ushy-clone} and \eqref{final-BasicBound}. Therefore we may instead assume there are at most $m-2$ integers colored by $1$ in $[2m-1-\mu-\nu,z]$.

In view of \eqref{ends-monochromatic-start} and \eqref{forgetton}, there are at least $m-\alpha+\beta+\frac{2m-4}{3}+1+\nu\geq m-\alpha+1$ integers colored by $1$ in $[2m-1-\mu-\nu,z]$.
Furthermore, in view of \eqref{forgetton}, we know that there are $\alpha+\beta$ integers colored by $0$ less than $2m-1-\mu-\nu$ with at most $m-1-\alpha$ integers colored by $1$ between $\first(0,R_1)$ and $2m-1-\mu-\nu$. Consequently, there exists a monochromatic $m$-subset $C'\subseteq [\first(0,R_1), z]$ with $2m-2\leq \diam C' \leq 2m-2+(m-1-\alpha)=2m-2+\beta'$. Indeed, simply take $C$ as defined in the previous paragraph and replace $\min C$ with the maximal integer colored by $0$ such that the resulting $m$-set $C'$ has $\diam C'\geq 2m-2$.
This integer exists as $|[\first(0,R_1),z]|\geq |\Delta^{-1}(0)\cap [1,z]|+|\Delta^{-1}\cap [\first(0,R_1),z]|\geq m+\alpha+\beta+m-\alpha+1\geq 2m+\beta+1$.
But now $A_0$, $C'$ and $D$ form a monochromatic solution to $p(m,m,m;2)$ in view of \eqref{ushy-clone} and \eqref{final-BasicBound}, completing the proof.
\end{proof}

\section{Acknowledgement}

This research was partially supported by NSF grant DMS0097317.

\end{document}